\documentclass[envcountsect]{llncs}
\usepackage[utf8]{inputenc}
\usepackage{amssymb,amsmath,mathrsfs}
\usepackage[english]{babel}
\usepackage[unicode,unicode=true,bookmarks=false]{hyperref}
\usepackage[shortlabels]{enumitem}
\usepackage{mathtools}

\usepackage[backend=biber]{biblatex}
\usepackage{csquotes}
\usepackage{xurl}

\usepackage{environ}

\newcommand{\repeattheorem}[1]{%
	\begingroup
	\renewcommand{\thetheorem}{\ref{#1}}%
	\expandafter\expandafter\expandafter\theorem
	\csname reptheorem@#1\endcsname
	\endtheorem
	\endgroup
}

\NewEnviron{reptheorem}[1]{%
	\global\expandafter\xdef\csname reptheorem@#1\endcsname{%
		\unexpanded\expandafter{\BODY}%
	}%
	\expandafter\theorem\BODY\unskip\label{#1}\endtheorem
}

\newcommand{\RR}{\mathbb{R}}
\newcommand{\NN}{\mathbb{N}}
\newcommand{\ZZ}{\mathbb{Z}}
\newcommand{\QQ}{\mathbb{Q}}
\newcommand{\PP}{\mathbb{P}}

\newcommand{\PrA}{\mathop{\mathbf{PrA}}\nolimits}
\newcommand{\PA}{\mathop{\mathbf{PA}}\nolimits}

\makeatletter

\renewcommand{\epsilon}{\varepsilon}
\renewcommand{\phi}{\varphi}
\newcommand{\sref}[2]{\hyperref[#2]{#1 \ref*{#2}}}
\newcommand{\dref}[2]{\hyperref[#2]{ #1 }}

\newcommand{\eqdef}{\stackrel{\mbox{\tiny\rm def}}{=}}\renewcommand{\to}{\rightarrow}
\newcommand{\ra}{\rightarrow}

\newcommand{\Lra}{\Leftrightarrow}

\newcommand{\degr}{\mathsf{deg}}

\spnewtheorem{hyp}{Conjecture}[section]{\bfseries}{\itshape}
\spnewtheorem{ex}{Example}{\bfseries}{\itshape}

\usepackage{color}
\newcommand{\fcomm}[1]{{\color{blue} Fedor: #1}}

\begin{document}
	\author{Fedor Pakhomov$^{12}$ and Alexander Zapryagaev$^3$}
	\title{Linear Orders in Presburger Arithmetic}
	\institute{${}^1$Department of Mathematics, University of Ghent, Krijgslaan 297-S8, 9000 Ghent, Belgium\\ ${}^2$Steklov Mathematical Institute of Russian Academy of Sciences, Gubkina 8, 119991 Moscow, Russia\\ ${}^3$Laboratory of Theoretical Computer Science, National Research University Higher School of Economics, 20, Myasnitskaya Str., 101000 Moscow, Russia}
	
	\maketitle
	
	\begin{abstract}
		We prove the linear orders first-order definable in the standard model $(\ZZ;<,+)$ of Presburger arithmetic are exactly those that are $(\ZZ;<,+)$-definably embeddable into the lexicographic ordering on $\ZZ^n$ for some $n$.
	\end{abstract}
	
	\section{Introduction}
	
	Presburger arithmetic $\PrA$ is the elementary theory of integers with order and addition. Unlike Peano Arithmetic $\PA$, it is complete, decidable and admits quantifier elimination in the extension of its language by constants $0,1$ and binary relations of congruence modulo $p$, $p\ge 2$ \cite{presburger}.
	
	
        A linear order $L$ is called definable in a structure $M$ if both its domain $D_L$ and its comparison relation $<_L\,\subseteq D_L\times D_L$ are first-order $M$-definable sets. Definability of linear orders (or of specific classes of linear orders) is a natural problem that has been previously studied for various structures and classes of structures (\cite{braudcarayol,khoussainovrubinstephan,ramakrishnan}). We aim to explore the definability of linear orders in Presburger arithmetic.

	Presburger arithmetic is not a sequential theory, i.\,e., it cannot encode tuples of integers by single integers. As such, it requires distinction between one-dimensional definability (each element of the order is represented by a single integer) and multi-dimensional definability (each element of the order is represented by a tuple of $m$ integers for a fixed $m\ge 2$).

    \emph{Scattered} orders are orders without dense suborders. A linear order has the Hausdorff rank $\le \alpha$ ($\alpha\in\mathsf{Ord}$) if it is embeddable into $\ZZ^\alpha \cdot n$ for some finite ordinal $n$ (slightly different definitions of rank could be found in the literature, but they lead to the value of rank varying at most by one). The classical result of Hausdorff is that a linear order is scattered iff it is of rank $\alpha$, for some $\alpha$.
        
    In a preceding paper \cite{pakhomovzapryagaev2020} the authors proved that any linear order definable in $(\ZZ,<,+)$ is a \emph{scattered} linear order of finite Hausdorff rank. 
	
	\begin{reptheorem}{ordering}[\cite{pakhomovzapryagaev2020}]
		Suppose a linear order $L$ is $m$-dimensionally definable in $(\ZZ,<,+)$. Then $L$ is scattered, and the Hausdorff rank of $L$ is at most $m$.
    \end{reptheorem}

    However, not all orders of finite Hausdorff rank are definable in $(\ZZ,<,+)$: there are continuum many non-isomorphic linear orders of rank $1$, and only at most countably many of them are definable in $(\ZZ,<,+)$. So \sref{Theorem}{ordering} only gives a necessary condition for a linear order to be definable in $(\ZZ,<,+)$. In this paper we give a complete characterization of the class of linear orders definable in $(\ZZ,<,+)$:
    
	\begin{reptheorem}{main}
		Any linear order definable in $(\ZZ,<,+)$ is $(\ZZ,<,+)$-definably embeddable into the lexicographic ordering on $\ZZ^n$.
    \end{reptheorem}

    This, of course, implies that any definable linear order in $(\ZZ,<,+)$ is isomorphic to a restriction of $\ZZ^n$ to a Presburger-definable set.

    In fact, the same holds in all models of Presburger arithmetic:
        \begin{reptheorem}{non-standard}
            If $M\models \PrA$, then any linear order $L$ definable in $M$ is $M$-definably embeddable into the lexicographic ordering on $M^n$.
        \end{reptheorem}
        
	
	
	\sref{Theorem}{non-standard} may be compared to results about certain classes of $o$-minimal structures establishing the definable embeddability of definable linear orders into the lexicographic order on the underlying structure \cite{onshuussteinhorn,ramakrishnan}. So despite not being $o$-minimal, with regard to definability of linear orders Presburger arithmetic has properties analogous to $o$-minimal theories.
	
	
	The work is organized as follows. Section 2 introduces the necessary background in Presburger arithmetic. In Section 3, we discuss linear orders with finite Hausdorff rank. Section 4 introduces piecewise polynomial functions and section cardinality functions. In Section 5, we introduce the notion of definable $N$-connectedness and prove a result about uniform definable $N$-connectedness of a definable family of individually connected sets, which is of interest in its own right. In Section 6, we prove \sref{Theorem}{main}.
	
	\section{Presburger arithmetic}
	This section gives necessary background on Presburger arithmetic.
	
	\begin{definition}
        \emph{Presburger arithmetic} $\PrA$ is the elementary theory of the structure $(\ZZ;<,+)$.
	\end{definition}
	
	Note that constants $0$, $1$ and the modulo comparison predicates $x\equiv_p y$ for all $p\ge 2$ are definable in $(\ZZ,<,+)$. In the extended language $(0,1,+,<,\{\equiv_p\}_{p\in\PP})$, Presburger arithmetic admits quantifier elimination \cite{presburger}. $\PrA$ is complete and decidable.
	
	Sometimes in the literature, Presburger arithmetic is considered as the elementary theory of the structure $(\NN;+)$ instead of $(\ZZ;<,+)$. These two structures are definitionally equivalent. Indeed, consider the bijection $f\colon\ZZ\to\NN$ mapping each negative integer $a$ to $-2a+1$ and each non-negative integer $a$ to $2a$. It is easy to see that the relation $(\NN,+)\models f(x)+f(y)=f(z)$ is definable in $(\ZZ;<,+)$, while the relations $(\ZZ,+)\models f^{-1}(x)<f^{-1}(y)$ and $(\ZZ,+)\models f^{-1}(x)+f^{-1}(y)=f^{-1}(z)$ are definable in $(\NN,+)$.
	
	
	
	For almost the entirety of this paper we will be working with the definable sets, functions, linear orders, etc. in $(\ZZ,<,+)$. Hence when we do not specifically mention where some set (or other object) is definable, we mean definability in $(\ZZ,<,+)$. A useful characterization of definable sets is provided by \sref{Theorem}{fund} below. It is obtained by combining the results of Ginsburg and Spanier \cite{ginsburgspanier} with those of Ito \cite{ito}.
	
	\begin{definition}
		For vectors $\vec{c},\vec{p_1},\ldots,\vec{p_n}\in\ZZ^m$ we call the set $J=\{\vec{c}+\sum k_i\vec{p_i}\mid k_i\in\NN\}\subseteq \ZZ^m$ a \emph{lattice} (or \emph{linear set}) generated by $\{\vec{p_i}\}$ from $\vec{c}$. If $\{\vec{p_1},\ldots,\vec{p_n}\}$ are linearly independent, we call this set an $n$-dimensional \emph{fundamental lattice}. A fundamental lattice of dimension $0$ is a single point $\vec{c}\in\ZZ^m$.
	\end{definition}
     
	The vectors $\vec{p_1},\ldots,\vec{p_n}$ are called the \emph{generating} vectors of $J$, and $\vec{c}$ its \emph{origin}. For a vector $\vec{y}\in J$, if $\vec{y}=\vec{c}+k_1\vec{p_1}+\ldots+k_n\vec{p_n}$, we say that the values $(k_1,\ldots,k_n)$ are the \emph{coordinates} of $\vec{y}$ in $J$.
	
	If $\vec{y}$ is expressed in the form $\vec{q}+l_1\vec{p_1}+\ldots+l_n\vec{p_n}$ for some vector $\vec{q}\in J$, we call $(l_1,\ldots,l_n)$ \emph{the coordinates of $\vec{y}$ with respect to $\vec{q}$} (if $\vec{q}$ is not the origin, $l_i$ might be negative).
        
    \begin{definition}
        We call a subset of $\ZZ^n$ a \emph{semilinear set} if it is a disjoint union of finitely many fundamental lattices.
    \end{definition}
        	
	\begin{theorem}[\cite{ginsburgspanier,ito}]\label{fund}
        A subset of $\ZZ^n$ is $\PrA$-definable iff it is a \emph{semilinear set}.
	\end{theorem}

    \begin{definition}
        For a fundamental lattice $J$ generated by $\vec{p}_1,\ldots,\vec{p}_n$ from $\vec{c}$, we call a function $f\colon J\to \ZZ$ \emph{linear} if $$f(\vec{c}+x_1\vec{p}_1+\ldots+x_n\vec{p}_n)=a_0+a_1x_1+\ldots+a_nx_n$$ for some $a_0,\ldots,a_n\in\ZZ$.
    \end{definition}
        
    \begin{definition}
        For a definable set $A\subseteq \ZZ^n$, a function $f\colon A\to \ZZ$ is called \emph{piecewise linear} if there is a decomposition of $A$ into a disjoint union of finitely many fundamental lattices $A=J_1\sqcup\ldots\sqcup J_k$ such that the restriction of $f$ to each $J_i$ is linear.
    \end{definition}

    \begin{theorem}[{\cite[Theorem~2.2]{pakhomovzapryagaev2020}}]\label{jetz}
	A function $f\colon A\ra\ZZ$ where $A\subseteq \ZZ^n$ is definable iff it is piecewise linear.
	\end{theorem}
        \begin{proof}
		The definability of all piecewise linear functions in Presburger arithmetic is obvious. A function $f\colon A\to\ZZ$ is definable if and only if its graph
		\begin{center}
		$G=\{(a_1,\ldots,a_n,f(a_1,\ldots,a_n))\mid (a_1,\ldots,a_n)\in A\}$
	\end{center}
	\noindent
	is definable. We represent $G$ as a union of fundamental lattices $G_1\sqcup\ldots\sqcup G_k$. For each $1\le i\le k$ let $G_i'$ be the projection of $G_i$ along the last coordinate: $$G_i'\eqdef\{(a_1,\ldots,a_n)\mid\exists a_{n+1}\left((a_1,\ldots,a_n,a_{n+1})\in G_i\right)\}.$$

    Notice that the projections of the generating vectors of $G_i$ along the last coordinate have to form a linearly independent set of vectors. Indeed, assume some non-trivial linear combination of them (with coefficients from $\QQ$) sums up to $0$. We first modify it by multiplying the combination by an appropriate positive natural number, to ensure that all the coefficients are integers. Then we consider some some linear combination of the same vectors with sufficiently large natural number coefficients and the sum of this new combination and the previous one. Thus we obtain two different linear combinations of the projections of the generating vectors of $G_i$ summing up to the same element of $G_i'$. But then the corresponding linear combinations of the generating vectors of $G_i$ would be producing two different vectors with the same projection along the last coordinate, contradicting the fact that $G$ was a graph of a function.
        
        Thus, all $G_i'$ are fundamental lattices. Clearly, the restriction of the function $f$ on any $G_i'$ is linear.
	\qed\end{proof}

	\begin{lemma}\label{expressive}
          For any fundamental lattice  $A=\{\vec{c}+\sum k_i\vec{p_i}\mid k_i\in\NN\}$  the functions returning coordinates $k_1,\ldots,k_n$ of elements $y=\vec{c}+k_1\vec{p_1}+\ldots+k_n\vec{p_n}$ are definable.
    \end{lemma}
    \begin{proof}
        This function is the inverse of the bijection $$f\colon \NN^n\to A,\;\;\;\;\;\;f\colon (k_1,\ldots,k_n)\longmapsto \vec{c}+k_1\vec{p_1}+\ldots+k_n\vec{p_n}.$$
        The lemma follows, since $f$ is clearly definable, and the inverse of a definable bijection is always definable.
    \qed\end{proof}
	
	The notion of a dimension for definable sets originates, as far as we are aware, from \cite{cluckers}, where it is shown that this notion is well-defined and the definition indeed assigns the unique dimension to each infinite definable set:
	
	\begin{definition}
        The \emph{dimension} $\dim(A)$ of an infinite definable set $A\subseteq\ZZ^m$ is the unique $k\ge 1$ such that there is a definable bijection between $A$ and $\NN^k.$
	\end{definition}

	It is easy to show it is equal to the largest $m\in\NN$ such that there exists an $m$-dimensional fundamental lattice which is a subset of $A$.
	
	The properties of Presburger dimensions are explored further in \cite[Section 3.1]{pakhomovzapryagaev2020}. We can extend the notion of a dimension to finite sets by thinking of them as sets of dimension $0$.
    
	
	
	In a more general setting, we consider definable \emph{parametric families} of sets.

    \begin{definition}
		For a set $A\subseteq \ZZ^{n+m}$ and $a\in\ZZ^n$, we define the \emph{section} $$A\upharpoonright a=\{b\in\ZZ^m\mid a^\frown b\in A\},$$ where for $(x_1,\ldots,x_n)\in\ZZ^n$ and $(y_1,\ldots,y_m)\in\ZZ^m$ $$(x_1,\ldots,x_n)^\frown (y_1,\ldots,y_m)\eqdef(x_1,\ldots,x_n,y_1,\ldots,y_m)\in\ZZ^{n+m}.$$
	\end{definition}
	
	\begin{definition}\label{prde}
		Given a definable set $P\subseteq \ZZ^n$, we say that a family $\langle A_p\mid p\in P\rangle$ of sets $A_p\subseteq \ZZ^m$ is \emph{(uniformly) definable} if  there is a definable set $A\subseteq P\times \ZZ^m$ such that $A_p=A\upharpoonright p$ for each $p\in P$.
	\end{definition}
	
	Similarly, we may consider parametric definable families of various kinds of objects, such as linear orders, by encoding them as sets. For example, for a definable set $P$, a definable family of linear orders $\langle L_p\mid p\in P\rangle$ on subsets of $\ZZ^m$ is given by two uniformly definable families of sets $\langle <_p\ \subseteq \ZZ^{2m}\mid p\in P\rangle$ and $\langle D_p\subseteq \ZZ^m\mid p\in P\rangle$, where for each $p\in P$,  $<_p\ \subseteq D_p\times D_p$, and $<_p$ is a linear order relation on $D_p$.

        Our standard order on $n$-tuples of integers is the lexicographical ordering: $$(x_0,\ldots,x_{n-1})<_{lex}(y_0,\ldots,y_{n-1})\Lra\exists i< n\:(\forall j<i\:(x_j=y_j)\wedge x_{i}<y_{i}).$$
        For $(x_1,\ldots,x_n)\in\ZZ^n$, the notation $|(x_1,\ldots,x_n)|$ will stand for $|\cdot|_{\infty}$-norm of $(x_1,\ldots,x_n)$, i.e. $$|(x_1,\ldots,x_n)|=\max(|x_1|,\ldots,|x_n|).$$

        Presburger arithmetic has \emph{definable choice functions}. That is, for any definable family $\langle A_p\subseteq \ZZ^n \mid p\in P\rangle$ of non-empty sets there is a definable function $f\colon P\to \ZZ^n$ such that $f(p)\in A_p$. Indeed, we could pick $f$ that maps $p$ to the $<_{lex}$-least element of $\{a\in A_p\mid |a|=m\}$, where $m$ is the least natural such that there exists $a\in A_p$ with $|a|=m$. 

        Similarly, for a definable equivalence relation $\sim$ on a definable set $A$, we are able to construct a definable set of representatives $A'\subseteq A$ such that $A'$ contains exactly one representative from each equivalence class. Specifically, we define $A'$ to comprise those $a\in A$ that for any $a'\in A$ if $a'\sim a$, then $a\le_{lex} a'$.

	\section{Scattered linear orders of finite rank}
	
	
	\begin{definition}
		A linear order $L$ is called \emph{scattered} (\cite[pp.\,32--33]{rosenstein}) if it does not have an infinite dense suborder.
	\end{definition}
	
	\begin{definition}
          Let $L$ be a linear order. We consider the equivalence relation $\simeq_L$ on $L$, where $a\simeq_L b$ if the cardinality of $\{t\in L\mid a<t<b\text{ or }b<t<a\}$ is finite. The \emph{condensation} $\mathsf{Cnd}(L)$ of $L$ is the order $L/{\simeq_L}$. We say that $L$ is an order of (Hausdorff) rank $n$ if the order $\mathsf{Cnd}^n(L)$ is finite, but for any $m<n$, the order $\mathsf{Cnd}^m(L)$ is infinite.
        \end{definition}

        We note that this definition of classes of orders of finite rank clearly agrees with the definition of classes of orders of rank $\alpha$ for ordinals $\alpha$.

        If $L$ is a linear order definable in $(\ZZ,<,+)$, then the relation $\simeq_L$ is Presburger-definable: $$a\simeq_L b\iff \exists n\forall c\in L(a<_L c<_L b\lor b<_L c<_La \to |c|\le n).$$
        Given $L$, we may identify $\mathsf{Cnd}(L)$ with some definable suborder of $L$ resulting from choosing exactly one representative from each $\simeq_L$-equivalence class in a definable manner.
        
        Clearly, for a definable family of linear orders $\langle L_p\mid p\in P\rangle$, the equivalences $\{\simeq_{L_p}\}$ form a definable family of binary relations. Hence, $\langle \mathsf{Cnd}(L_p)\mid p\in P\rangle$ is a definable family of suborders.

	In \cite{pakhomovzapryagaev2020}, we proved the rank condition:
	
	\repeattheorem{ordering}
	
        \begin{corollary}
            In any definable family of linear orders $\langle L_p\mid p\in P\rangle$, the ranks of $\{L_p\}$ are bounded by some $n\in\NN$.
        \end{corollary}
        \begin{proof}
          The dimensions of domains of $L_p$ are bounded by some natural $n$, hence, by Theorem \ref{ordering}, the ranks of $L_p$ themselves are bounded by $n$.
        \qed\end{proof}
        
        \section{Piecewise polynomial functions}
    \begin{definition}\label{polynomial}
          Let $J\subseteq \ZZ^n$ be a fundamental lattice generated by vectors $\vec{p}_1,\ldots,\vec{p}_m$ from the origin $\vec{c}$. We call a function $F\colon J\to \ZZ$ \emph{polynomial} if there is a polynomial with rational coefficients $P(x_1,\ldots, x_m)$ such that $F(\vec{c}+\vec{p}_1x_1+\ldots+\vec{p}_mx_m)=P(x_1,\ldots, x_m)$ for all $(x_1,\ldots,x_m)\in\NN^m$.
        \end{definition}
        
        We note that if $F$ is a polynomial function on a fundamental lattice $J$, then the polynomial $P$ as above is uniquely determined. Here it is clearly suffices to prove that for every non-zero polynomial $P\in Q[x_1,\ldots,x_{n}]$ there are naturals $a_1,\ldots,a_n$ such that $P(a_1,\ldots,a_n)\ne 0$. We establish this by induction on $n$. The base case of $n=0$ is trivially true. For $n>0$ we first put $P$ in the form $x_n^0P_0+\ldots +x_n^mP_m$, where $P_0,\ldots,P_m\in Q[x_1,\ldots,x_{n-1}]$. Then we pick a non-zero $P_i$ and by induction assumption find naturals $a_1,\ldots,a_{n-1}$ such that $P_i(a_1,\ldots,a_{n-1})\ne 0$. Finally we note that for at least one of the values $a_n\in\{0,\ldots,m\}$, $P(a_1,\ldots,a_n)\ne 0$, since otherwise the non-zero polynomial of one variable $P(a_1,\ldots,a_{n-1},x)$ would have more roots than its degree.
        
        \begin{definition}\label{piecewise-polynomial}
          Let $A\subseteq \ZZ^n$ be a definable set. We call a function $F\colon A\ra\ZZ$ \emph{piecewise polynomial} if there is a decomposition of $A$ into finitely many fundamental lattices $J_1\sqcup\ldots\sqcup J_k$ such that the restriction of $F$ onto each $J_i$ is a polynomial function. The degree $\degr(f)$ of $F$ is the maximum of degrees of the restrictions.
    \end{definition}

        Just after the proof of \sref{Lemma}{growth} we will show that the degree of a piecewise polynomial $F$ does not depend on the decomposition of $A$ into a disjoin union of the fundamental lattices and thus that the degree of a piecewise polynomial function is well-defined.

        First let us note that the reason we are interested in piecewise polynomial functions is primarily the following technical result from \cite{pakhomovzapryagaev2020}, which in turn was based on a theorem of Blakley \cite{blakley}.
    \begin{lemma}[{\cite[Corollary~5.1]{pakhomovzapryagaev2020}}]\label{card_funct}
        For any definable family of finite sets $$\langle A_p\subseteq \ZZ^n\mid p\in P\rangle$$ the cardinality function $p\mapsto|A_p|$ is piecewise polynomial with degree $\le n$.
    \end{lemma}

    Recall that for functions $f,g\colon \NN\to \NN$ we write $f(x)=O(g(x))$ whenever there exists such a constant $M>0$ that $|f(x)|<Mg(x)$ for all $x\in\NN$; we write $f(x)=o(g(x))$ if $\lim_{x\ra\infty}\frac{f(x)}{g(x)}=0$. We will use Donald~Knuth's notation $f(x)=\Theta(g(x))$ to express that simultaneously $f(x)=O(g(x))$ and $g(x)=O(f(x))$ hold.

    With any piecewise polynomial function $F\colon A\to \ZZ$ we associate the following monotone function $h_F\colon\NN\ra\NN$:
        $$h_F(n)=\max \{|F(a)|\mid |a|\le n\}.$$

        \begin{lemma}\label{composition} Suppose $g\colon A\to B$ is a definable function and $F\colon B\to \ZZ$ is a piecewise polynomial function such that $h_F(x)=O(x^k)$. Then $h_{F\circ g}(x)=O(x^k)$.\end{lemma}
        \begin{proof}Clearly, the function $u\colon \NN\to\NN$, $u\colon x \mapsto \max\{g(y)\mid |y|\le x\}$ is definable. Hence, $u$ is semilinear and is thus bounded from above by some linear function $b x+a$, $a,b\in\NN$. Therefore $h_{F\circ g}(x)\le h_{F}(bx+a)$, and hence $h_{F\circ g}(x)=O(x^k)$.\qed\end{proof}

    \begin{lemma}\label{growth}
        Let $J$ be a fundamental lattice and $F\colon J\ra\ZZ$ be a polynomial function of the degree $k$. Then $h_F(x)=\Theta(x^k)$.
        \end{lemma}
        \begin{proof}
          Consider the polynomial $P(x_1,\ldots,x_n)\in\QQ[x_1,\ldots,x_n]$ of degree $k$ assigned to $F$ in accordance with \sref{Definition}{polynomial}. Clearly, there is a definable function $g\colon \NN^n\to J$ such that $F(a)=P(g(a))$ for all $a\in\NN^n$. In view of \sref{Lemma}{composition}, it suffices to prove that $h_P(x)=\Theta(x^k)$. Trivially, $h_P(x)=O(x^k)$, so it remains to show that $x^k=O(h_P(x))$.

          Consider the polynomial $Q(x_1,\ldots,x_n)$ equal to the sum of all monomials of the degree $k$ that occur in $P$. We fix some point in $\vec{s}\in[0,1]^n$ such that $Q(\vec{s}\,)\ne 0$ (here by $[0,1]$ we mean the interval in $\RR$). Clearly, $Q(x\vec{s}\,)=x^kQ(\vec{s}\,)$ and also $P(x\vec{s}\,)-Q(x\vec{s}\,)=O(x^{k-1})$. Hence, $P(x\vec{s}\,)=Q(\vec{s}\,)\cdot x^k(1+o(1))$.

          For each $i\in\NN$, let $\vec{a}_i\in\NN^n$ be the vector made of the integer parts of the numbers $i\cdot\vec{s}$. Obviously, $P(\vec{a}_x)-P(x\vec{s}\,)=O(x^{k-1})$. Hence $P(\vec{a}_x)=Q(\vec{s}\,)\cdot x^k(1+o(1))$. Since we chose $\vec{s}$ such that $Q(\vec{s}\,)\ne 0$, we have $|P(\vec{a}_x)| = |Q(\vec{s}\,)|\cdot x^k(1+o(1))$.
          
          Finally, as $|\vec{a}_x|\le x$, we obtain $h_{P}(x)\ge  |Q(\vec{s}\,)|\cdot x^k(1+o(1))$ and thus $x^k=O(h_P(x))$. \qed\end{proof}

        Now let us prove that the degree of a piecewise polynomial function is well-defined. Consider a piecewise polynomial function $F\colon A\to \ZZ$ and some decomposition of $A$ into finitely many fundamental lattices $A=J_1\sqcup\ldots\sqcup J_n$ such that the restrictions $F_i$ of $F$ to $J_i$ are polynomial. We see that $h_F(x)=\max_{1\le i\le n}h_{F_i}(x)$. Hence $h_F(x)=\Theta(x^k)$, where $k$ is the maximum of the degrees of all $F_i$. Thus, the degree of $F$ can be recovered from $h_F$ alone and hence does not depend on the choice of the particular decomposition $A=J_1\sqcup\ldots\sqcup J_n$.

        The same reasoning also gives us the following generalization of \sref{Lemma}{growth}:
        
        \begin{lemma}\label{growth2}
        A function $F\colon A\ra\ZZ$ is piecewise polynomial of degree $k$ iff $h_F(x)=\Theta(x^k)$.
    \end{lemma}

        The following two lemmas are obtained immediately from \sref{Lemma}{growth2}.
        
    \begin{lemma}\label{bu}
        Suppose $F\colon A\ra\ZZ$ and $G\colon A\ra\ZZ$ are piecewise polynomial functions, such that $|F(x)|\le |G(x)|$, for all $x\in A$. Then $\degr(F)\le\degr(G)$.
    \end{lemma}
    
    \begin{lemma}\label{estimate}
        Suppose $F,G_1,\ldots,G_n$ are piecewise polynomial functions on $A$, such that $|F(x)|\le\max\limits_{1\le i\le n}|G_i(x)|$, for all $x\in A$. Then $\degr(F)\le \max\limits_{1\le i\le n}\degr(G_i)$.
    \end{lemma}

        Combining \sref{Lemma}{composition} with  \sref{Lemma}{growth2}, we obtain another result:
        \begin{lemma}\label{composition2} For any definable function $g\colon A\to B$ and piecewise polynomial $F\colon B\to \ZZ$, the degree of $F\circ g$ is less than or equal to the degree of $F$.\end{lemma}

        We shall also prove a useful
        
        \begin{lemma}\label{difference} Suppose $F\colon J\to \ZZ$ is a polynomial function of degree $k$ and $g_1,g_2\colon A\to J$ are definable functions such that for $x\in A$, a bound $|g_1(x)-g_2(x)|\le N$ holds for some positive $N$. Then $G(x)=F(g_1(x))-F(g_2(x))$ is a piecewise polynomial function with degree $<k$.\end{lemma}
        \begin{proof}
            The fact that the $G(x)$ is piecewise polynomial is trivial. In light of \sref{Lemma}{growth2} it remains to verify that $h_G(x)=O(|x|^{k-1})$.
            
            As in \sref{Definition}{polynomial}, we consider polynomial $P$ such that $F(\vec{c}+\vec{p}_1x_1+\ldots+\vec{p}_mx_m)=P(x_1,\ldots, x_m)$ for all $(x_1,\ldots,x_m)\in\NN^m$, where $\vec{p}_1,\ldots\vec{p}_m$ are the generating vectors of $J$. Naturally, $P$ is of degree $k$. Next we find definable functions $f_1,f_2\colon A\to \NN^m$ that are produced from $g_1$ and $g_2$ by mapping to the internal coordinates of $J$. Thus $G(x)= P(f_1(x))-P(f_2(x))$, and for $x\in A$ there is a bound $|f_1(x)-f_2(x)|\le N'$, where $N'$ is some fixed number.

           Since both functions $f_1$ and $f_2$ are definable, we have $|f_1(x)|=O(|x|)$ and $|f_2(x)|=O(|x|)$. Since  $|f_1(x)-f_2(x)|\le N'$, the differences $f_1(x)-f_2(x)$ take only finitely many possible values $d\in D$. Hence it suffices to show that for each fixed $d\in \ZZ^n$ we have $P(y)-P(y+d)=O(|y|^{k-1})$, where $y$ ranges over $\ZZ^m$. But it is a standard fact that $P(y)-P(y+d)$ is a polynomial of degree $\le k-1$, which completes the proof.
        \end{proof}

	\section{Definable connectedness}

\begin{definition}
    We say that a set $A\subseteq \ZZ^n$ is \emph{$N$-connected}, for a positive integer $N$, if for any two points $a,b\in A$ there exists a finite sequence $\langle c_i\in A\mid 0\le i\le k\rangle$ such that $c_0=a$, $c_k=b$, and $|c_{i+1}-c_i|\le N$ for each $i$, $0\le i<k$.
\end{definition}

We shall refer to such a sequence $\{c_i\}$ as \emph{$N$-bounded}.

\begin{definition}
    We say that a definable family of sets $\langle A_p\mid p\in P\rangle $ is \emph{definably $N$-connected} if there exists a definable family of sequences $\langle \langle c^{a,b,p}_i\in A_p \mid 0\le  i\le k^{a,b,p}\rangle \mid p\in P,\ a,b\in A_p\rangle$ such that for each $p\in P$ and $a,b\in A_p$ we have $c^{a,b,p}_0=a$, $c^{a,b,p}_{k^{a,b,p}}=b$, and $|c^{a,b,p}_{i+1}-c^{a,b,p}_i|\le N$ for each $i$, $0\le i<k^{a,b,p}$.
\end{definition}

The goal of this section is to prove the following result:
\begin{theorem}\label{connect}
    Suppose $\langle A_p \mid p\in P\rangle$ is a definable family of sets such that each individual $A_p$ is $N$-connected for the same value of $N$. Then, for a suitable positive integer $N'$, the family $\langle A_p \mid p\in P\rangle$ is definably $N'$-connected.
\end{theorem}

The key lemma that we are going to prove is:
\begin{lemma}\label{defconn}
   For any fundamental lattice $J\subseteq \ZZ^{m+n}$, the family of sections $\langle J\upharpoonright p \mid p\in \ZZ^m\rangle$ is definably $N$-connected for a sufficiently large $N$.
\end{lemma}


    In order to obtain the result of this lemma, we introduce the following technical definition.
    
	\begin{definition}
		Let $J\subseteq\ZZ^q$ be a fundamental lattice generated by the vectors $v_1,\ldots,v_l$ from the origin $r$ and $k$ a natural number, $0\le k\le q$. We define a partial function $\pi_k^J\colon J\times \ZZ^k\ra J$ that we shall call \emph{the definable projection of $a\in J$ onto $J\upharpoonright p$} as follows. To determine the value of $\pi_k^J(a,p)$, we consider all possible vectors of the form $$r+x_1\cdot v_1+\ldots+x_l\cdot v_l\in J\upharpoonright p\text{, where }x_1,\ldots,x_l\in\NN.$$ We call vectors of this form the \emph{candidates} for the value of $\pi_k^J(a,p)$. The value of $\pi_k^J(a,p)$ is the candidate for which the vector $(x_1,\ldots,x_l)$ is lexicographically minimal. If there are no candidates, $\pi_k^J(a,p)$ is undefined.
	\end{definition}
	
	Informally, we consider the coordinates of the candidates \emph{as seen from $a$} instead of the true origin $r$ of $J$ (maybe negative) and pick the point for which the relative coordinates are both non-negative and lexicographically minimal among all non-negative coordinates. According to \sref{Lemma}{expressive}, the property of being a candidate for $\pi_k^J(a,p)$  is definable. Hence, the existence of $\pi_k^J$ and its value whenever it exists are definable.

	\begin{lemma}\label{Nacc}
		For any fundamental lattice $J$, there is a sufficiently large $N$ such that for each generating vector $v_i$ of $J$, $a\in J$, and $p\in \ZZ^k$, if both $\pi_k^J(a,p)$ and $\pi_k^J(a+v_i,p)$ are defined, then $|\pi_k^J(a,p)-\pi_k^J(a+v_i,p)|\le N$.
	\end{lemma}
	\begin{proof}
		It suffices to find such $N$ for a fixed generating vector $v_i$ and then take the maximum over $i$.
		
		We consider the (nonempty) set $U$ of all vectors $(u_1,\ldots,u_l)\in\ZZ^l$ such that the first $k$ coordinates of the vector $u_1 v_1+\ldots+u_l v_l+1\cdot v_i$ are equal to $0$.
        
        We split $U$ into fundamental lattices $U_1,\ldots,U_t$ in such a way that each lattice preserves all coordinate signs, that is, if for some $(u_1,\ldots,u_l)\in U_j$ some particular $u_c$ is positive (zero, negative), then for all elements of $U_j$ the corresponding coordinate is positive (zero, negative). In order to achieve that, we split $U$ definably into its subsets for each possible combination of signs first, and then divide each of those subsets individually into fundamental lattices.
        
        Let $r_j=(r_{j,1},\ldots,r_{j,l_{j}})$ be the origins of the individual lattices $U_j$. For a vector $d=(d_1,\ldots,d_{l_j})$, we will use $d\vec{v}$ as a shorthand for $d_1 v_1+\ldots+d_{l_{j}} v_{l_{j}}$.
		
		We claim that $N:=\max\limits_{1\le j\le t}|r_j\vec{v}+v_i|$ suffices for our purpose. Consider some $a,p$ such that both $\pi_k^J(a,p)$ and $\pi_k^J(a+v_i,p)$ are defined and let $s\in \NN^l$ be such that $\pi_k^J(a,p)=a+s\vec{v}$. Observe that  $\pi_k^J(a+v_i,p)=\pi_k^J(a,p)+d\vec{v}+v_i=a+s\vec{v}+d\vec{v}+v_i$, for some $d\in U$.

        If we now show that $d=r_j$ for some $j$, we are done. Assume for a contradiction that $d$ belongs to a particular $U_j$, but $d\neq r_j$.
		
		The vectors $d$ and $r_j$ lie in the same $U_j$, and thus all pairs of matching coordinates $d_l$ and $r_{j,l}$ share the same sign. Also, as $r_j$ is the origin of $U_j$, $|d_l|\ge |r_{j,l}|$ for each $l$: if, for example, $0<d_l<r_{j,l}$, then $r_j+k\cdot(d-r_j)\in U_j$ would have a negative $l$th coordinate $r_{j,l}+k(d_l-r_{j,l})$ for a sufficiently large $k$, contradicting the sign retention property of $U_j$.
		
		Now we prove the following: 
        
        \begin{claim}
            \begin{enumerate}
                \item $a+(s+d-r_j)\vec{v}=\pi_k^J(a,p)+d\vec{v}-r_j\vec{v}$ is a candidate for the value of $\pi_k^J(a,p)$,
                \item $a+v_i+(s+r_j)\vec{v}=\pi_k^J(a+v_i,p)-d\vec{v}+r_j\vec{v}$ is a candidate for the value of $\pi_k^J(a+v_i,p)$.
            \end{enumerate}
        \end{claim}
        \begin{proof}
            As $d$ and $r_j$ both belong to $U$, the vector $d\vec{v}-r_j\vec{v}$ lies in $J\upharpoonright\vec{0}$. Hence, both $\pi_k^J(a,p)+(d-r_j)\vec{v}$ and $\pi_k^J(a+v_i,p)-(d-r_j)\vec{v}$ belong to $J\upharpoonright p$. It remains to check that the relative coordinates of these points as seen from $a$ and $a+v_i$ respectively are non-negative.

            The coordinates $s$ are non-negative, as they are the relative coordinates of $\pi_k^J(a,p)$ seen from $a$. Similarly, as $\pi_k^J(a+v_i,p)=a+v_i+s\vec{v}+d\vec{v}$ is the projection of $a+v_i$, the relative coordinates $s+d$ of $\pi_k^J(a+v_i,p)$ as seen from $a+v_i$ are non-negative.

            First, we check the coordinates of $\pi_k^J(a,p)+(d-r_j)\vec{v}=a+(s+d-r_j)\vec{v}$ as seen from $a$, namely, the values of $s+d-r_j$. Indeed: consider a particular coordinate $s_l+d_l-r_{j,l}$. Exactly one of the following holds:

            \begin{enumerate}
			\item $d_l>0$. In this case, $0<r_{j,l}\le d_l$, hence $0\le s_l\le s_l+d_l-r_{j,l}$;
			\item $d_l<0$. In this case, $d_l\le r_{j,l}<0$, hence $-1\le s_l+d_l<s_l+d_l-r_{j,l}$;
			\item $d_l=0$. In this case, $r_{j,l}=0$, and $0\le s_l=s_l+d_l-r_{j,l}$.
		\end{enumerate}

        In each of these options, $s_l+d_l-r_{j,l}\ge 0$. This shows that $a+(s+d-r_j)\vec{v}$ is a valid candidate for $\pi_k^J(a,p)$.

        Similarly, in order to prove $a+v_i+(s+r_j)\vec{v}$ is a candidate for $\pi_k^J(a+v_i,p)$, we need to show that the coordinates $v_i+(s+r_j)\vec{v}$ are a non-negative linear combination of the vectors in $\vec{v}$. We shall do it separately depending on whether $l=i$.
        
        First, let us fix some specific $l$-th coordinate $l\ne i$. Exactly one of the following possibilities holds for $s_l+r_{j,l}$:

        \begin{enumerate}
			\item $d_l>0$. In this case, $0<r_{j,l}\le d_l$, hence $0\le s_l<s_l+r_{j,l}$;
			\item $d_l<0$. In this case, $d_l\le r_{j,l}<0$, hence $0\le s_l+d_l\le s_l+r_{j,l}$;
			\item $d_l=0$. In this case, $r_{j,l}=0$, and $0\le s_l=s_l+r_{j,l}$.
        \end{enumerate}

        In each of these options, $s_l+r_{j,l}\ge 0$. Now we consider the $i$-th coordinate, which equals $s_i+r_{j,i}+1$. We have the following possibilities:
        
                \begin{enumerate}
			\item $d_i>0$. In this case, $0<r_{j,i}\le d_i$, hence $0\le s_i<s_i+r_{j,i}+1$;
			\item $d_i<0$. In this case, $d_i\le r_{j,i}<0$, hence $0\le s_i+d_i+1\le s_i+r_{j,i}+1$;
			\item $d_i=0$. In this case, $r_{j,i}=0$, and $0\le s_i<s_i+r_{j,i}+1$.
                \end{enumerate}

                In each of these options, $s_i+r_{j,i}+1\ge 0$. Thus, $a+v_i+(s+r_j)\vec{v}$ is a valid candidate for $\pi_k^J(a+v_i,p)$.
        \qed\end{proof}
        It remains to see that either 

        \begin{displayquote}
			$a+(s+d-r_j)v$ is lexicographically smaller than $a+s\vec{v}$ in the ordering among the candidates for $\pi_k^J(a,p)$,
		\end{displayquote}
        \noindent
        or
        \begin{displayquote}
            $a+v_i+(s+r_j)\vec{v}$ is lexicographically smaller than $a+v_i+(s+d)\vec{v}$ in the ordering of the candidates for $\pi_k^J(a+v_i,p)$,
        \end{displayquote}
        \noindent
        but not simultaneously. This is due to the fact that the differences between elements of the compared pairs are $(d-r_j)\vec{v}$ and $-(d-r_j)\vec{v}$. Thus, when we consider the lexicographic ordering, the results of the comparison should be mutually opposite. 
        
        Hence, at least one of the values $\pi_k^J(a,p)$ or $\pi_k^J(a+v_i,p)$ was not the lexicographically minimal candidate, as mandated by the definition of $\pi_k^J(-,p)$, contradiction.   
	\qed\end{proof}	

    \emph{Proof of \sref{Lemma}{defconn}.} Let $x$ and $y$ be two arbitrary points of $J$, let $a$ be the origin of $J$, and let $v_1,\ldots,v_n$  be its generating vectors. It is easy to define a sequence of points in $J$ that connects $a$ to $x$ by steps along the generating vectors: for $x=a+x_1 v_1+\ldots+x_n v_n$, the sequence is $$\begin{aligned}a,\;\; a+v_1,\;\; a+2v_1,\ldots,\;\;a+x_1 v_1,\;\; a+x_1 v_1+v_2,\ldots,\;\;a+x_1v_1+x_2v_2,\\\ldots,\\a+x_1v_1+\ldots+x_{n-1}v_{n-1}+v_n,\ldots,\;\; a+x_1v_1+\ldots+x_nv_n.\end{aligned}$$ Similarly, we define a path from $a$ to $y$. By reversing the first path and concatenating the two, we produce a path from $x$ to $y$ by single steps along the generating vectors of $J$.
			
	We note that, from the definition of $\pi_k^{J}$, for each element $r$ along this path, the value of $\pi_k^{J}(r,p)$ is defined. Indeed, $x$ is a valid candidate for the value of $\pi_k^{J}(r,p)$ for each $r$ along the path from $x$ to $a$, and $y$ is a valid candidate for the value of $\pi_k^{J}(r,p)$ for each $r$ along the path from $a$ to $y$.
			
	By replacing each point $r$ with its own projection $\pi_k^{J}(r,p)$, we construct a new path connecting $x$ and $y$ that now lies fully in $J\upharpoonright p$. Applying \sref{Lemma}{Nacc}, we find a large $N$ (not dependent on $p$) such that the distances between each two sequential points on this path are bounded by $N$. This is due to the fact that, before applying the projections, each step followed some generating vector of $J$.\qed

    \emph{Proof of \sref{Theorem}{connect}.} Let $A^1\sqcup A^2\sqcup\ldots\sqcup A^s$ be a decomposition of $\langle A_p\mid p\in P\rangle$ (treated as a definable subset of $P\times\ZZ^n$) into fundamental lattices. According to \sref{Lemma}{defconn}, each family $\langle A^i \upharpoonright p \mid p\in P\rangle$, $1\le i\le s$, is definably $N_i$-connected for some $N_i$. We set $N'=\max\{N_1,N_2,\ldots,N_s,N\}$, which is sufficient, and show that the family $\langle A_p \mid p\in P\rangle$ is definably $N'$-connected.

    Let $p\in P$ be fixed, and let $a,b\in A_p$ be two arbitrary points. Assume further that $a$ and $b$ lie in sections $A^k\upharpoonright p$ and $A^l\upharpoonright p$, respectively, $1\le k,l\le s$. We are going to define a path from $a$ to $b$ so that the distances between the neighboring points are bounded by $N$. 

    We call a pair of sections $A^u\upharpoonright p$, $A^v\upharpoonright p$ $N$-neighboring if there exists a pair of points $c^{u,v}_p\in A^u\upharpoonright p$ and $d^{u,v}_p\in A^v\upharpoonright p$ such that $|c^{u,v}_p-d^{u,v}_p|\le N$. The property that, for $p\in P$, the pair of sections $A^u\upharpoonright p$, $A^v\upharpoonright p$ is $N$-neighboring, is clearly definable. Using definable choice, we fix the pairs $c^{u,v}_p,d^{u,v}_p$ for all the cases when the pair of $A^u\upharpoonright p$ and $A^v\upharpoonright p$ is $N$-neighboring. Since $A_p$ is $N$-connected, for each $p$ the graph of $A^i\upharpoonright p$ under $N$ is $N$-connected. Observe that we can pick, in a definable way, a sequence of distinct indices  $k=w_1,w_2,\ldots,w_q=l$ such that $A^{w_i}\upharpoonright p$ and $A^{w_{i+1}}\upharpoonright p$ are $N$-neighboring for $1\le i<q$: here we use the fact that there are only finitely many such sequences possible.

 We construct the definable path from $a\in A^u\upharpoonright p$ to $b\in A^v\upharpoonright p$ by concatenating the following chain of paths: $$a\ra c^{w_1,w_2}_p, d^{w_1, w_2}_p\ra c^{w_2,w_3}_p, d^{w_2,w_3}_p \ra c^{w_3,w_4}_p,\ldots, d^{w_{q-1},w_q}_p\ra b,$$ where the all the individual paths are constructed via an application of \sref{Lemma}{defconn}, while the steps from $c^{w_i,w_{i+1}}_p$ to $d^{w_i,w_{i+1}}_p$ are fixed as above and bounded by $N$.
 
 It remains to notice that the whole construction provides for a definable family of paths.\qed
	
	\section{Embeddability of linear orders}
	
    In this section we prove our main result on the classification of definable linear orders in Presburger arithmetic.	
	\repeattheorem{main}
	
        In order to prove this theorem, we will obtain an even stronger statement:
        
        \begin{theorem}\label{general}For any \emph{definable family} of linear orderings $\langle L_p\mid p\in P\rangle$, there exists a \emph{definable family} of embeddings of linear orderings $\langle f_p\colon L_p\ra \ZZ^n\mid p\in P\rangle$, for some $n$. 
        \end{theorem}
	\begin{proof}
		First we consider the case where all $L_p$ are finite.
		
		\begin{lemma}\label{1}
			Let $\langle L_p\mid p\in P\rangle$ be a definable family of finite linear orderings. Then for some $n\ge 1$ there is a definable family of embeddings of linear orderings $\langle f_p\colon L_p\ra \ZZ^n\mid p\in P\rangle$.
		\end{lemma}
		\begin{proof}
			By \sref{Lemma}{card_funct}, the function $S\colon p\longmapsto|L_p|$ is piecewise polynomial. We prove the lemma by induction on the degree $m$ of $S$.
			
			If $m=0$, then the range of $S$ is a finite set. Hence, the sizes of all $L_p$ are simultaneously bounded by some fixed natural number, and the maps $f_p\colon L_p\to\ZZ$ can be constructed explicitly. Now we prove the induction step.
			
			We will set up a certain definable family of finite sequences $\langle a_{p,i}\in L_p\mid p\in P, 1\le i\le o_p\rangle$ such that $a_{p,1}$ is the minimum of $L_p$ and $a_{p,o_p}$ is the maximum. The key property of the sequence will be that the degree of the piecewise polynomial function $H(p,i)$, which maps the pair $(p\in P,1<i\le o_p)$ to the number of points between $a_{p,i}$ and $a_{p,i+1}$ (inclusive) within $L_p$, will be strictly less than $m$.

           Having found such a sequence, the rest of the proof will proceed as follows. We consider the subsequence
           $\langle a_{p,i}\in L_p\mid p\in P, i\in I_p\rangle$, where $I_p$ consists precisely of those indices $1\le i\le o_p$ such that $a_{p,j}<_{L_p} a_{p,i}$, for all $1\le j<i$.  We consider the following splitting of $L_p$ into a family of its non-intersecting intervals $S_{p,i}$ indexed by $i\in I_p$:
           \begin{enumerate}
               \item $S_{p,1}$ consists of the single point $a_{p,1}$,
               \item $S_{p,i}$ is the interval $(a_{p,i^-},a_{p,i}]^{L_p}$, where $i^-=\max \{j\in I_p \mid j<i\}$.
           \end{enumerate}
           Clearly, for $p\in P$ and $i\in I_p\setminus \{1\}$, $|S_{p,i}|\le H(i)$. Therefore, following the construction of $H$, we can apply the induction hypothesis to the family of linear orders $\langle S_{p,i} \mid p\in P, i\in I_p\rangle$. 
           In this way we obtain a family of embeddings of linear orders $\langle  f'_{p,i}\colon S_{p,i} \to \ZZ^{n'} \mid p \in P, i\in I_p\rangle$. Then we set $n=n'+1$ and $f_p\colon L_p \to \ZZ^{n}$ to be the functions that, for each $p\in P$ and $i\in I_p$, send given $x\in S_{p,i}$ to  $(i)^\frown f'_{p,i}(x)$.

           In the following we shall construct the definable family of sequences $\langle a_{p,i}\in L_p\mid p\in P, 1\le i\le o_p\rangle$ with the desired properties.

          For all $p$ and $a\in L_p$, we consider the initial segments $$C_{p,a}=\{b\in L_p\mid b<_{L_p}a\}.$$ It is clear that $C_{p,a}$ is a definable family of finite sets and hence the corresponding cardinality function $R\colon p^\frown a\longmapsto|C_{p,a}|$ is piecewise polynomial. The degree $\degr(R)$ does not exceed $m$ by \sref{Lemma}{bu}, since it is bounded by a piecewise polynomial function $p^\frown a\longmapsto w(p)$, which is of degree $\le m$ by \sref{Lemma}{composition2}.
            
			Let $J_1,\ldots, J_k$ be the fundamental lattices such that the restriction of $R$ to each $J_l$ is polynomial.  We will define sequences $a_{p,i}$ in such a way that for each $p\in P$ and $1\le i<o_p$, the elements $a_{p,i}$ and $a_{p,i+1}$ are either in the same $J_l$ or they are neighboring in the order $L_p$. When $a_{p,i}$ and $a_{p,i+1}$ are from the same $J_l$, we shall furthermore guarantee that $|a_{p,i}-a_{p,i+1}|$ is bounded by a positive constant $N_l$ (independent of $p$ and $i$). This entails that, when restricted to such $p$ and $i$, the value $H(p,i)$ (the number of points between $a_{p,i}$ and $a_{p,i+1}$ in the sense of $L_p$) is equal to $1+|R(p\frown a_{p,i+1})-R(p\frown a_{p,i})|$. From \sref{Lemma}{difference}, we can see that the degree of $R(p\frown a_{p,i+1})-R(p\frown a_{p,i})$ is smaller than $m$. Hence the degree of $H$ is also smaller than $m$. From this, the whole function $(p,i)\mapsto |S_{p,i}|$ also has degree $<m$.

        For a given $p\in P$, we construct the sequence $\langle a_{p,i}\in L_p\mid 1\le i\le o_p\rangle$ as follows. We choose a sequence of distinct indices $1\le l_{p,1},\ldots,l_{p,s}\le k$ and elements $b_{p,j},c_{p,j}\in L_p$ for $1\le j\le s$ such that:
        \begin{enumerate}
            \item $b_{p,j}$ is the element of $L_p$ immediately following $\max_{L_p}(\bigcup_{1\le t<j} J_{l_{p,t}}\upharpoonright p)$;
            \item $l_{p,j}$ is the unique index such that $b_{p,j}\in J_{l_{p,j}}\upharpoonright p$;
            \item $c_{p,j}=\max_{L_p}(J_{l_{p,j}}\upharpoonright p)$;
            \item $s=j$ when $c_{p,j}$ is the greatest element of $L_p$.
        \end{enumerate}
        
        Now $\langle a_{p,i}\in L_p\mid, 1\le i\le o_p\rangle$ is the concatenation of the sequences $b_{p,1}\ra c_{p,1},\ldots,b_{p,s}\ra c_{p,s}$, where each individual sequence is obtained as a result of application of \sref{Lemma}{defconn}. The constants $N_{l_j}$ are provided by the statement of \sref{Lemma}{defconn}. We see that the sequence $\langle a_{p,i}\in L_p\mid 1\le i\le o_p\rangle$ obtained in this way has all the required properties: it starts with the minimum of $L_p$, ends with its maximum, and for each  pair of neighboring elements of the sequence they are either neighbors in $L_p$ or they belong to the same $J_j$ and have a distance of at most $N_{j}$.

\end{proof}
		
		\begin{lemma}\label{2}
			Let $L_p$ be a definable parametric family of  linear orderings such that for all $p\in P$, for all $L_p$, for all $a,b\in L_p$, intervals $[a,b]$ are finite. Then there is exists definable parametric family of embeddings $f_p\colon L_p\ra\ZZ^n$.
		\end{lemma}
		\begin{proof}                  
            Since empty orders are embeddable everywhere by the empty embedding, we may assume that the family $L_p$ does not contain empty orders. 

            We will construct uniformly definable sequences $\langle c_{p,s}\in L_p\mid s\in\ZZ\rangle$. We fix a definable choice function $u$ mapping $p$ to a specific element of $L_p$ (say, the lexicographically smallest one) and set $c_{p,0}=u(p)$. For all $s>0$, we define $$c_{p,s}=\max_{L_p}(\{c_{p,0}\}\cup\{a\in L_p\mid|a|\le s\}),$$ while for all $s<0$, we define $$c_{p,s}=\min_{L_p}(\{c_{p,0}\}\cup\{a\in L_p\mid|a|\le s\}).$$

            Clearly, $c_{p,s}$ is a non-decreasing sequence in $L_p$ such that every $a\in L_p$ that is not the maximal element of $L_p$ lies in some $L_p$-interval $[c_{p,s},c_{p,s+1})=R_{p,s}$. If $L_p$ has a maximal element, we add it at the top of the last non-empty $R_{p,s}$. Clearly, $\langle R_{p,s}\mid p\in P, s\in \ZZ\rangle$  is a definable family of finite linear orders.

            We apply \sref{Lemma}{1} to $\langle R_{p,s}\mid p\in P, s\in \ZZ\rangle$ and obtain a uniformly definable family  of embeddings $g_{p,s}\colon R_{p,s}\to \ZZ^{n'}$. Now the statement of the lemma follows for $n=n'+1$ and $f_p\colon  L_p\ra\ZZ^n$ mapping each $a\in R_{p,s}$ to $s^{\frown}g_{p,s}(a)$.
		\qed\end{proof}
		
		\begin{lemma}\label{3}
			Let $L_p$ be a definable parametric family of linear orderings. Then there exists a definable parametric family of embeddings $f_p\colon L_p\ra\ZZ^n$.
		\end{lemma}
		\begin{proof}
			The proof proceeds by induction on the maximal rank of $L_p$ in the family. \sref{Lemma}{2} provides the base case. For the induction step, consider the condensations $\mathsf{Cnd}(L_p)$.
            
            By induction hypothesis, there exists a uniformly definable family of embeddings $g_p\colon \mathsf{Cnd}(L_p)\ra\ZZ^{n_1}$. Now we consider the family of orders $$\langle U_{p,t}\mid p\in P,t\in\mathsf{Cnd}(L_p)\rangle$$ which are restrictions of $L_p$ to elements $a\in L_p$ such that $a\simeq_{L_p}t$. We apply \sref{Lemma}{2} to this family and get a family of embeddings $h_{p,t}\colon U_{p,t}\ra \ZZ^{n_2}$. Now the statement of the lemma follows for $n=n_1+n_2$ and $f_p\colon L_p\to \ZZ^n$ mapping $a\in L_p$ to $g_p(t)^\frown h_{p,t}(a)$, where $t$ is the unique element of $\mathsf{Cnd}(L_p)$ that is $\simeq_{L_p}$-equivalent to $a$.\qed\end{proof}
    
    Combining the results of \sref{Lemma}{1}, \sref{Lemma}{2}, and \sref{Lemma}{3}, we complete the proof of \sref{Theorem}{general}.\qed\end{proof}

        \repeattheorem{non-standard}
        \begin{proof}
          Let $L$ be $k$-dimensionally definable with parameters $\vec{a}\in M^m$ by formulas $\varphi(x_1,\ldots,x_k,\vec{a})$ and $\psi(x_1,\ldots,x_k,y_1,\ldots,y_k,\vec{a})$ defining the domain of the order and the comparison relation, respectively. Consider the formula $\chi(\vec{z})$ such that for $\vec{a}'\in M^m$, the sentence $\chi(\vec{a}')$ expresses: ``formulas $\varphi(x_1,\ldots,x_k,\vec{a}')$ and $\psi(x_1,\ldots,x_k,y_1,\ldots,y_k,\vec{a}')$ define a linear order on $k$-tuples".

          In the standard model $(\ZZ,<,+)$, consider the definable family of linear orders $\langle L_{\vec{a}'}\mid \vec{a}'\in P\rangle$, where $P=\{\vec{a}'\mid (\ZZ,<,+)\models \chi(\vec{a}')\}$ and $L_{\vec{a}'}$ is a linear order defined by  $\varphi(x_1,\ldots,x_k,\vec{a}')$ and $\psi(x_1,\ldots,x_k,y_1,\ldots,y_k,\vec{a}')$. This is indeed a definable family of linear orders since $(\ZZ,<,+)$ and $M$ are elementary equivalent. By \sref{Theorem}{general}, there exists a definable family of embeddings $\langle f_p\colon L_p\to \ZZ^n\mid p\in P\rangle$. Let the formula $\theta(x_1,\ldots,x_k,y_1,\ldots,y_n,\vec{z})$ define this family of embeddings.

          From the elementary equivalence of $(\ZZ,<,+)$ and $M$ we see that the formula $\theta(x_1,\ldots,x_k,y_1,\ldots,y_n,\vec{a})$ gives an $M$-definition of an embedding of $L$ into $M^n$.
        \qed\end{proof}
        
        We note that a definable order of the rank $\le n$ is not necessarily definable $n$-dimensionally. For example, order $$1^{m-1}+\ZZ+2^{m-1}+\ZZ+3^{m-1}+\ZZ+\ldots$$ of rank $2$ is clearly embeddable in $\ZZ^m$. It does not possess any $m'$-dimensional definable isomorphic copy $L$, for $m'<m$.
        
        Indeed, assume for a contradiction that there is an $m'$-dimensional definable isomorphic copy $L$ of the order described above. We fix the definable suborder $S\subseteq\mathsf{Cnd}(L)$ consisting of those elements $s\in\mathsf{Cnd}(L)$ for which the corresponding $\simeq_L$-equivalence class is finite. Next, we consider the definable families of finite sets $\langle I_i\mid i\in \NN\rangle$ and $\langle A_i\mid i\in \NN\rangle$, where $$I_i=\{s\in S \mid |s|\le i\}\text{ and }A_i=\bigcup_{s\in I_i} [s]_{\simeq_L}.$$
        Since $S$ is infinite by Lemma \ref{card_funct}, the function $f\colon i\mapsto |I_i|$ is piecewise polynomial of degree $k\le 1$ and, as $L$ was $m'$-dimensional, by the same lemma the function $g\colon i\mapsto |A_i|$ is piecewise polynomial of degree $l\le m'<m$. Since $f$ is monotone by Lemma \ref{growth2}, we observe that for some $\alpha>0$ and $N$, for all $i\ge N$ we have $f(i)\ge \alpha i$. Now, the union of any $x$ distinct finite $\simeq_L$-equivalence classes has the cardinality at least $1^{m-1}+2^{m-1}+\ldots+x^{m-1}=\Theta(x^m)$. Since for $i\ge N$, $g(i)$ is the cardinality of the union of at least $\alpha i$ distinct finite $\simeq_L$-equivalence classes, we obtain $g(i)\ge \beta i^m$ for some $\beta>0$. Hence  $g$ has degree $\ge m$, contradiction.
        
        However, it is not currently clear whether any $n$-dimensionally definable order can be definably embeddable into $\ZZ^m$, for some $m$.
        
        \begin{question}What is the pointwise minimal function $t\colon \NN\to\NN$ such that, for all $n\in\NN$, any $n$-dimensionally definable order $L$ is embeddable into $\ZZ^{t(n)}$?\end{question}
        
        By examining our proof of \sref{Theorem}{general}, it is easy to see that if an order $L$ is $n$-dimensionally definable in $(\ZZ,<,+)$, then it will be definably embeddable into $\ZZ^{(n+1)^2}$. It is likely that this estimate can be substantially improved.

	\section*{Acknowledgments}
	
	The authors thank Lev~Beklemishev for the suggestion of a research topic and many fruitful discussions.
	
	The work of Fedor Pakhomov is supported by FWO grant G0F8421N. The work of Alexander~Zapryagaev was prepared within the framework of the HSE University Basic Research Program.

    \emergencystretch=2em
	
	\printbibliography[heading=bibintoc,title={References}]

\end{document}